\documentclass[11pt,a4paper]{article}
\usepackage{amssymb}
\usepackage{amsthm}
\usepackage{amsmath}
\usepackage{amsfonts}
\usepackage[mathscr]{eucal}
\usepackage{graphicx}
\usepackage{multicol}

\let\phi\varphi

\newtheorem{definition}{Definition}[section]
\newtheorem{cor}{Corollary}[section]
\newtheorem{lemma}{Lemma}[section]
\newtheorem{remark}{Remark}[section]
\newtheorem{theorem}{Theorem}[section]
\newtheorem{proposition}{Proposition}[section]
\newtheorem{condition}{Condition}[section]

\makeatletter
\@addtoreset{equation}{section}
\@addtoreset{theorem}{section}

\title{On the nonlocal stabilization by starting control
of the normal equation generated from Helmholtz system\thanks{The research of the first author was supported by the Ministry of Education
and Science of the Russian Federation (grant 14.Z50.31.0037). The
second author was supported by RFBR grants 15-01-03576 and
15-01-08023. This paper has been accepted for publication in SCIENCE CHINA Mathematics.}} %Use the shortened version of the full title

\author{A.\, V.\, Fursikov\\ \footnotesize{Department of Mechanics \& Mathematics, Moscow State University,} \\ \footnotesize{Moscow, 119991
Russia;}\\ \footnotesize{Voronezh State University, Voronezh, 394018,  Russia,}
\and L.S. Osipova \\\footnotesize{Department of Mechanics \& Mathematics, Moscow State University,} \\ \footnotesize{Moscow, 119991
Russia}}

\date{}
\begin{document}

\maketitle

\begin{abstract}
We consider the problem of stabilization to zero of semilinear
normal parabolic equations connected with the 3D Helmholtz system
with periodic boundary conditions and arbitrary initial datum.
This problem was previously studied in \cite{FSh16}. As it was
recently revealed, the control function suggested in that work
contains a term impeding transference the stabilization
construction on the 3D Helmholtz system. The main concern of this
article is to prove that this term is not necessary for the
stabilization result, and therefore the control function can be
changed by a proper way.
\end{abstract}


\section{Introduction}
This work is connected with construction of nonlocal stabilization
of solutions for equations of hydrodynamic type by feedback control
\footnote{the term "nonlocal" means that the
distance between steady-state solution and initial condition of
stabilized solution can be of arbitrary magnitude}.
There exists extensive literature on the local stabilization of
Navier-Stokes system in a small neighborhood of a stationary point
(see for example, \cite{BLT}, \cite{B}, \cite{F6}, \cite{FG}, \cite{R1}, \cite{RTh}
as well as literature listed in the review \cite{FK})  but
construction of its  nonlocal analog is in the initial stage yet.
Note that for some equations of fluid dynamics there
are certain nonlocal stabilization results: for Burgers equation,
where exact formula of its solution was used (see \cite{K}), and
for Euler equations (see \cite{C1},\cite{C2}), where the
construction is based on such properties of its solutions which
Navier-Stokes system does not possess. We have to note also that
nonlocal exact controllability of the Navier-Stokes system by
distributed control supported in a sub domain of the spatial
domain where this system is defined  has been proved in \cite{CF}
for 2D case and in \cite{FI} for 3D case. Since settings of exact
controllability and stabilization problems are related in some
sense, this gives us the hope that nonlocal stabilization problem
can be solved, because settings of exact controllability and stabilization
problems are related in some sense.

Let discuss the setting of nonlocal stabilization problem near zero for
3D Navier-Stokes system with periodic boundary conditions written in abstract form:
\begin{equation}\label{NS_stab1}
      \partial v(t,x)+NS(v)(t,x)=\sum_{j=1}^N\delta (t-t_j)u_j(x),\quad  v(t,x)|_{t=0}=v_0(x).
\end{equation}
 Here $v(t,x)$ is its (unknown) solution, $v_0\in V^1(\mathbb{T}^3)$\footnote{definition of the space $V^1$ see below at \eqref{phase_space}}  is a given initial datum,
 $\sum_{j=1}^N\delta (t-t_j)u_j(x)$ is unknown impulse feedback control where $\delta(t-t_j)$
  is Dirac $\delta$ function at $t_j$ and  $t_j, u_j$ are defined with
 $v(t_j,\cdot)$. We assume that $u_j(x)\in V^1(\mathbb{T}^3)$ and $\mbox{supp}u_j\subset \mathcal D$ where a given  sub domain $\mathcal D \subset \mathbb{T}^3 $ does not depend on $j$.

Formulation of stabilization problem for \eqref{NS_stab1} is as follows:

Given $v_0$ find control $\sum_{j=1}^N\delta (t-t_j)u_j(x)$ such that
\begin{equation}\label{NS_stab2}
    \| v(t,\cdot)\|_1\le c\mathrm{e}^{-t},\quad \mbox{as}\;\; t\to \infty\quad \mbox{with}\;\; c=c(\|v_0\|_1)
\end{equation}
where $\| \cdot \|_1$ is the norm of the space $V^1(\mathbb{T}^3)$.

It is very important that here we have to look for solution $v$ of problem \eqref{NS_stab1},\eqref{NS_stab2}
 in the class of smooth enough functions where uniqueness theorem for 3D Navier-Stokes equations has been proved, because stabilization problem can be considered for dynamical systems only. Recall that millennium problem for 3D Navier-Stokes equations is to prove just in such function class the existence of solution
 for boundary value problem connected with 3D Navier-Stokes system, and this problem is not solved yet.
 Thus, there is some connection between millennium problem and problem \eqref{NS_stab1},\eqref{NS_stab2},
 and progress in solution of the last one can be useful to understand the difficulties connected with the first problem better.

 For solution of problem \eqref{NS_stab1},\eqref{NS_stab2} it is more convenient to go over
 Navier-Stokes equation \eqref{NS_stab1} for fluid velocity $v$ to Helmholtz equation for $\mbox{curl}\,v$:
 this allows to go over phase space $V^1$ to the more convenient phase space $V^0$. Further, to solve local stabilization problem all authors begin with studying its main linear part. In the case of non local stabilization we also do the same. But if in local case the main part  is linearization of stabilization problem, in nonlocal case this is so-called nonlocal stabilization problem by starting control for normal parabolic equation (NPE) generated from three-dimensional Helmholtz system.
 \footnote{Let us explain why here we use starting control for stabilization. It is established in local stabilization theory (see, for instance \cite{F6},\cite{FG}, \cite{FK} and references therein) that one can construct impulse and distributed (i.e. realized by external forces) controls by means of some set of starting controls. Moreover, when equations of hydrodynamics type are defined on a domain $G$, but not on a torus, one can construct the control defined on the boundary $\partial G$ with help of a set of starting ones. Recall that usually a boundary control is the most natural from physical point of view. We believe that in nonlocal stabilization theory situation will be similar, and therefore the approach of stabilization by starting control proposed here will be the first step in construction of general control theory.}

Theory of such problems was
first constructed for NPE associated with the differentiated Burgers equation (see \cite{F5},
\cite{FSh}), and after that for NPE associated with Helmholtz system in
\cite{FSh16}, \cite{F7}. The purpose of this article is the further development of this theory.

The non-local stabilization problem of NPE by starting control is
formulated as follows:

Given fixed $0<a_j<b_j<2\pi , j=1,2,3$, and divergence-free
initial condition $y_0(x)$ of NPE associated with the 3D Helmholtz
system with periodic boundary conditions \footnote{In other words independent variables $x=(x_1,x_2,x_3)$ run through $3D$
torus $\mathbb{T}^3=(\mathbb{R}/2\pi \mathbb{Z})^3$}, find a
divergence free control $u_0(x)$, supported on
$[a_1,b_1]\times[a_2,b_2]\times[a_3,b_3]\subset \mathbb{T}^3$,
such that the solution $y(t,x)$ of NPE with initial condition
$y_0+u_0$ satisfies the inequality
\begin{equation}
\|y(t,\cdot)\|_{L_2(\mathbb{T}^3)}\leq\alpha\|y_0+u_0\|_{L_2(\mathbb{T}^3)}\mathrm{e}^{-t},
\quad \forall t>0
\end{equation}
with some $\alpha>1$.

This problem was solved in \cite{FSh16}, \cite{F7}. Namely, it was proved, that the NPE with arbitrary
initial condition $y_0$ can be stabilized by starting control in the form
\begin{equation}\label{introduction_control}
u_0(x)=Fy_0-\lambda u(x),
\end{equation}
where $Fy_0$ is a certain feedback control with feedback operator
$F$, constructed by some technic of local stabilization theory
(see, for example, \cite{F6}, \cite{FG}), $\lambda>0$ is a
constant depending on $y_0$, and $u$ is a universal function,
depending only on the given parallelepiped
$[a_1,b_1]\times[a_2,b_2]\times[a_3,b_3]\subset \mathbb{T}^3$,
which contains the support of $u$ as well as support of the whole
control $u_0$. The proof of the stabilisation result is bazed on
the following estimate:
\begin{equation}\label{main_est_b0}
\int_{T^3}((\mathbf{S}(t,{x};{u}),\nabla)\operatorname{curl}^{-1}\mathbf{S}(t,{x};{u}),\mathbf{S}(t,{x};{u}))dx>3\beta
\mathrm{e}^{-18t} \qquad \forall \; t\ge 0,
\end{equation}
where $S(t,x;u)$ is the solution of the Stokes equation with
initial condition $u$, and $\beta>0$ is some constant.

The proof of estimate \eqref{main_est_b0} is very complicated and
was made in \cite{F5}, \cite{FSh}, \cite{FSh16}.

Our future goal is to develop nonlocal stabilization theory such
that it can be applied for the 3D Helmholtz system. The first
attempts to realize this plan were immediately shown that the term
$Fy_0$ from control function \eqref{introduction_control}, used in
\cite{FSh16}, \cite{F7}, does not allow to transfer stabilization
construction on the 3D Helmholtz system. The aim of this work is
to prove nonlocal stabilization of NPE by starting control that
does not contain the term $Fy_0$.

In section 2 we remind the definitions and some facts concerning
NPE associated with 3D Helmholtz system. Section 3 is devoted to
the proof of the main stabilization result for NPE, with help of
starting control \eqref{introduction_control} with omitted term
$Fy_0$.

%%%%%%%%%%%%%%%%%%%%%%%%%%%%%%%%%%%%%%%%%%%%%%%%%%%%%%%%%%%%%%%%%%%%%%%%%%%%%%%%%%%%%%%%%%%%%%%%%%%%%%%%%%%%%%%%%%%%%%%
%%%%%%%%%%%%%%%%%%%%%%%%%%%%%%%%%%%%%%%%%%%%%%%%%%%%%%%%%%%%%%%%%%%%%%%%%%%%%%%%%%%%%%%%%%%%%%%%%%%%%%%%%%%%%%%%%%%%%%%

\section{Semilinear parabolic equation of normal type}\label{s1}
In this section we recall the definition and basic properties of
normal parabolic equations corresponding to 3D Helmholtz system
with periodic boundary conditions: the explisit formula for their
solution, the theorem on the existence and uniqueness of solution
for normal parabolic equations and the structure of their
dynamics. These results have been obtained in \cite{F2}-\cite{F4}.
We begin with formulation of Navier-Stokes equations that are
basic in the theory of viscous incompressible fluid.

%%%%%%%%%%%%%%%%%%%%%%%%%%%%%%%%%%%%%%%%%%%%%%%%%%%%%%%%%%%

\subsection{Navier-Stokes equations}\label{s1.1}

Let us consider 3D Navier-Stokes system
\begin{equation}\label{NavierStokes}
      \partial_t {v}(t,{x})- \Delta {v}(t,{x})+({v},\nabla){v}+\nabla{p}(t,{x})=0, \, \operatorname{div} {v}=0,
\end{equation}
with periodic boundary conditions
%where $\partial_tv=\partial v/\partial t,\; \partial_{xx}v=\partial^2v/\partial x^2$ with periodic boundary condition
\begin{equation}\label{NavierStokes_boundcond}
   {v}(t,...,x_i,...)={v}(t,...,x_i+2\pi,...), \, i=1,2,3
\end{equation}
and initial condition
\begin{equation}\label{NavierStokes_incond}
  {v}(t,{x})|_{t=0}={v}_0({x})
\end{equation}
where $t\in \mathbb{R}_{+}$, ${x}=(x_1,x_2,x_3)\in\mathbb{R}^3$, ${v}(t,{x})=(v_1,v_2,v_3)$ is the velocity vector field of fluid flow, $\nabla p$ is the gradient of pressure, $\Delta$ is the Laplace operator, ${({v},\nabla){v}=\sum_{j=1}^{3}v_j\partial_{x_j}v}$. Periodic boundary conditions \eqref{NavierStokes_boundcond} mean that Navier-Stokes eqautions \eqref{NavierStokes} and initial conditions \eqref{NavierStokes_incond} are defined on torus $\mathbb{T}^3=(\mathbb{R}/2\pi\mathbb{Z})^3$.

For each $m\in \mathbb{Z}_+=\{j\in \mathbb{Z}:j\geq 0\}$ we define the space
\begin{equation}\label{phase_space}
V^m=V^m(\mathbb{T}^3)=\{v(x)\in(H^m(\mathbb{T}^3))^3:\mbox{div}v=0,\int_{\mathbb{T}^3}v(x)dx=0\}
\end{equation}
where $H^m(\mathbb{T}^3)$ is the Sobolev space of functions belonging to $L_2(\mathbb{T}^3)$ together with their derivatives up to the order $m$.

It is well-known, that the non-linear term $(v,\nabla)v$ in problem \eqref{NavierStokes}-\eqref{NavierStokes_incond} satisfies relation
$$\int_{\mathbb{T}^3}(v(t,x),\nabla)v(t,x)\cdot v(t,x)dx=0.$$

Therefore, multiplying \eqref{NavierStokes} scalarly by $v$ in $L_2(\mathbb{T}^3)$, integrating by parts by $x$,
and then integrating by $t$, we obtain the well-known energy estimate
\begin{equation}\label{en_est}
\int_{\mathbb{T}^3}|v(t,x)|^2dx+2\int_{0}^t\int_{\mathbb{T}^3}|\nabla_x v(\tau, x)|^2dx d\tau\leq\int_{\mathbb{T}^3}|v_0(x)|^2dx,
\end{equation}
which allows to prove the existence of weak solution for
\eqref{NavierStokes}-\eqref{NavierStokes_incond}. But, as is
well-known, scalar multiplication of \eqref{NavierStokes} by $v$
in $V^1(\mathbb{T}^3)$ does not result into an analog of estimate
\eqref{en_est}. Nevertheless, expression of such kind will be
important for us when they will be written in equivalent form with
help of solutions of Helmholtz system.

\subsection{Helmholtz equations}
Using problem
\eqref{NavierStokes}-\eqref{NavierStokes_incond} for fluid
velocity $v$, let us derive the similar problem for the curl of velocity
\begin{equation}\label{curl}
{\omega}(t,x)=\operatorname{curl}v(t,x) = (\partial_{x_2}v_3 - \partial_{x_3}v_2,\partial_{x_3}v_1 - \partial_{x_1}v_3,\partial_{x_1}v_2 - \partial_{x_2}v_1)
\end{equation}

It is well-known from vector analysis, that
\begin{equation}\label{vectan1}
(v,\nabla)v = \omega\times v+\nabla\frac{|v|^2}{2},
\end{equation}
\begin{equation}\label{vectan2}
\operatorname{curl}(\omega\times v)=(v,\nabla)\omega-(\omega,\nabla)v, \text{ if } \operatorname{div}v=0,\, \operatorname{div}\omega=0.
\end{equation}
where $\omega\times v = (\omega_2v_3 - \omega_3 v_2,\omega_3v_1 -
\omega_1v_3,\omega_1v_2 - \omega_2v_1)$ is the vector product of
$\omega$ and $v$, and $|v|^2=v_1^2+v_2^2+v_3^2$. Substituting
\eqref{vectan1} into \eqref{NavierStokes} and applying curl
operator to both sides of the obtained equation, taking into
account \eqref{curl}, \eqref{vectan2} and formula
$\operatorname{curl}\nabla F=0$, we obtain the Helmholtz equations
\begin{equation}\label{Helmholtz}
\partial_t\omega(t,x)-\Delta \omega+(v,\nabla)\omega-(\omega,\nabla)v=0
\end{equation}
with initial conditions
\begin{equation}\label{Helmholtz_incond}
\omega(t,x)|_{t=0}=\omega_0(x):=\operatorname{curl}v_0(x),
\end{equation}
and periodic boundary conditions.

\subsection{Derivation of Normal Parabolic Equations (NPE)}\label{s1.2}
Using decomposition into Fourier series
\begin{equation}\label{Fourier_decomp}
v(x)=\sum_{k\in\mathbb{Z}^3}\hat{v}(k)\mathrm{e}^{\mathrm{i}(k,x)}, \, \, \hat{v}(k)=(2\pi)^{-3}\int_{\mathbb{T}^3}v(x)\mathrm{e}^{-\mathrm{i}(k,x)}dx,
\end{equation}
where $(k,x)=k_1\cdot x_1+k_2\cdot x_2+k_3\cdot x_3$,
$k=(k_1,k_2,k_3)$, and the well-known formula
$\mbox{curl}\,\mbox{curl}\,v = -\Delta v$, if $\mbox{div }\,v =0$,
we see that inverse operator to curl is well-defined on space
$V^m$ and is given by the formula
\begin{equation}\label{rotinv}
\mbox{curl}^{-1}\omega(x)=\mathrm{i}\sum_{k\in\mathbb{Z}^3}\frac{k\times\hat{\omega}(k)}{|k|^2}\mathrm{e}^{\mathrm{i}(k,x)}.
\end{equation}
Using formulas $\mbox{curl}v=\omega ,\, \mbox{div}v=0$ one can get
by straightforward calculations that $\| \omega
\|^2_{L_2(\mathbb{T}^3)}=\| \nabla v\|^2_{L_2(\mathbb{T}^3)}$.
 Therefore, operator $\mbox{curl}:V^1\mapsto V^0$ realizes
isomorphism of the spaces, and it is a unitary operator. Thus, a
sphere in $V^1$ for
 \eqref{NavierStokes}-\eqref{NavierStokes_incond} is equivalent to a sphere in $V^0$
 for the problem \eqref{Helmholtz}-\eqref{Helmholtz_incond}.

Let us denote the non-linear term in Helmholtz system by $B$:
\begin{equation}\label{B_term}
B(\omega)=(v,\nabla)\omega-(\omega,\nabla)v,
\end{equation}
where $v$ can be expressed in terms of $\omega$ using
\eqref{rotinv}.

Multiplying \eqref{B_term} scalarly by $\omega=(\omega_1, \omega_2, \omega_3)$ and integrating by parts, we get expression
\begin{equation}\label{B_term_mult}
(B(\omega),\omega)_{{V^0}}=-\int_{\mathbb{T}^3}\sum_{j,k=1}^{3}\omega_j\partial_jv_k\omega_kdx,
\end{equation}
that, generally speaking, is not zero. Hence, energy estimate for
solutions of 3-D Helmholtz system is not fulfilled. In other
words, operator $B$ allows decomposition
\begin{equation}\label{B_term_decomp}
B(\omega) = B_n(\omega)+B_{\tau}(\omega),
\end{equation}
where vector $B_n(\omega)$ is orthogonal to sphere
$\Sigma(\|\omega\|_{V^0})=\{u\in V^0:\|u\|_{V^0}=\|\omega\|_{V^0}
\}$ at the point $\omega$, and vector $B_{\tau}$ is tangent to
$\Sigma(\|\omega\|_{V^0})$ at $\omega$. In general, both terms in
\eqref{B_term_decomp} are not equal to zero. Since the presence of
$B_n$, and not of $B_{\tau}$, prevents the fulfillments of the
energy estimate, it is plausible that the term $B_n$ by itself
generates the possible singularities in the solution. Therefore,
there is reason to omit the $B_{\tau}$ term in Helmholtz system
and study first the equations \eqref{Helmholtz} with non-linear
operator $B(\omega)$ replaced with $B_n(\omega)$. \footnote{I.e.
there is reason to begin investigation of stabilization problem
for Helmholtz system with studies of stabilization problem for
indicated equations.} The obtained equations will be called the
normal parabolic equations (NPE).

Let us derive the NPE corresponding to
\eqref{Helmholtz}-\eqref{Helmholtz_incond}.

Since summand $(v,\nabla)\omega$ in \eqref{B_term} is  tangential
to vector $\omega$, the normal part of operator $B$ is defined by
the summand $(\omega,\nabla)v$. We shall seek it in the form
$\Phi(\omega)\omega$, where $\Phi$ is the unknown functional,
which can be found from equation
\begin{equation}\label{Phi_eq}
\int_{\mathbb{T}^3}\Phi(\omega)\omega(x)\cdot\omega(x)dx=\int_{\mathbb{T}^3}(\omega(x),\nabla)v(x)\cdot \omega(x)dx.
\end{equation}
According to \eqref{Phi_eq},
\begin{equation}\label{Phi_def}
\Phi(\omega) = \left\{\begin{array}{rl}
\int_{\mathbb T^3}(\omega(x),\nabla)\operatorname{curl}^{-1}\omega(x)\cdot \omega(x)dx/\int_{\mathbb T^3}|\omega(x)|^2dx,
 & \omega\neq 0,\\
        0, & \omega \equiv 0.
            \end{array}\right.
\end{equation}
where $\operatorname{curl}^{-1}\omega(x)$ is defined in \eqref{rotinv}.

Thus, we arrive at the following system of normal parabolic equations corresponding to Helmholtz equations
\eqref{Helmholtz}:
\begin{equation}\label{npe}
\partial_t \omega(t,x) - \Delta \omega - \Phi(\omega)\omega = 0,\, \operatorname{div}\omega = 0,
\end{equation}
\begin{equation}\label{NPE_boundcond}
   {\omega}(t,...,x_i,...)={\omega}(t,...,x_i+2\pi,...), \, i=1,2,3
\end{equation}
where $\Phi$ is the functional defined in \eqref{Phi_def}

Further we study problem \eqref{npe}, \eqref{NPE_boundcond} with
initial condition \eqref{Helmholtz_incond}.

\subsection{Explicit formula for solution of NPE}\label{s2.2}
%%%%%%%%%%%%%%%%%%%%%%%%%%%%%%%

In this subsection we remind the explicit formula for NPE
solution.

\begin{lemma}\label{exp_sol_lem}
Let $\mathbf{S}(t,x;\omega_0)$ be the solution of the following Stokes system with periodic boundary conditions:

\begin{eqnarray}
&\partial_t z - \Delta z=0,\, \operatorname{div}z = 0;\label{3Dheat_eq}\\
& z(t,...,x_i+2\pi,...) = z(t,x),\,\,\,\, i=1,2,3;\label{3DS_boundary_cond}\\
&z(0,x)=\omega_0,\label{3Dheat_in_cond}
\end{eqnarray}
i.e. $\mathbf{S}(t,x;\omega_0)=z(t,x)$. \footnote{Note that
because of periodic boundary conditions the Stokes system should
not contain the pressure term $\nabla p$} Then the solution of
problem \eqref{npe} with periodic boundary conditions and initial
condition \eqref{Helmholtz_incond} has the form
\begin{equation}\label{exp_sol}
   \omega(t,x;\omega_0)=\frac{\mathbf{S}(t,x;\omega_0)}{1-\int_0^t\Phi (\mathbf{S}(\tau,\cdot ;\omega_0))d\tau}
\end{equation}
\end{lemma}

One can see the proof of this Lemma in \cite{F2}, \cite{F4}. %Note
%that in \cite{F2}, \cite{F4} analytical and geometrical structure
%of sets $M_-, M_+, M_g$ are studied. We do not formulate these
%results here because they will not be used below.

\subsection{Properties of the functional $\Phi(u)$}\label{s2.4}

Let us now study some properties of the functional $\Phi(u)$, defined in \eqref{Phi_def}.

For every $s\in\mathbb{R}$ the Sobolev space $H^{s}(\mathbb{T}^3)$ is defined as the space of periodic real distributions with finite norm
\begin{equation}\label{sobolev_space_norm}
\|z\|^2_{H^s(\mathbb{T})^3}\equiv\|z\|^2_s=\sum_{k\in\mathbb{Z}^3}|k|^{2s}|\hat{z}(k)|^2<\infty,
\end{equation}
where $\hat{z}(k)$ are the Fourier coefficients of function $z$. Note, that everywhere below we will be considering only functions $z$ with $\hat{z}(0)=0$.

We shall need the following space:

\begin{equation}\label{V_s_space}
V^s\equiv V^s(\mathbb{T}^3)=
\left\{v(x)\in (H^s(\mathbb{T}^3))^3: \operatorname{div} v(x)=0, \int_{\mathbb{T}^3}v(x)dx=0\right\},
s\in \mathbb{R}.
\end{equation}

\begin{lemma}\label{Phi_prop1_lem}
There exists a constant $c>0$ such that for every $u\in V^{3/2}$ functional $\Phi(u)$, defined in \eqref{Phi_def}, satisfies the following estimate:
\begin{equation}\label{Phi_prop1}
|\Phi(u)|\leq c\|u\|_{3/2}.
\end{equation}
\end{lemma}

\begin{proof}
According to Sobolev's embedding theorem,
$H^{1/2}(\mathbb{T}^3)\subset L_3(\mathbb{T}^3)$, therefore, using
definition \eqref{Phi_def} and interpolation inequality
$\|v\|^3_{1/2}\leq \|v\|_0^2\|v\|_{3/2}$, we get estimate
\eqref{Phi_prop1}:
\begin{equation*}
|\Phi(u)|\leq \frac{\|u\|^2_{L_3(\mathbb{T}^3)}\|\nabla \operatorname{curl}^{-1}u\|_{L_3(\mathbb{T}^3)}}{\|u\|^2_0}\leq
 c\frac{\|u\|^3_{1/2}}{\|u\|^2_0}\leq c\frac{\|u\|^2_0 \|u\|_{3/2}}{\|u\|^2_0}=c\|u\|_{3/2}.
\end{equation*}
\end{proof}

\begin{lemma}\label{Phi_prop2_lem}
For any $\beta<1/2$ there exists a constant $c_1>0$ such that %There exists a constant $c>0$ such that for every $u\in V^{3/2}$ functional $\Phi(u)$, defined in \eqref{Phi_def}, satisfies the following estimate:
\begin{equation}\label{Phi_prop2}
\left|\int_0^t\Phi(\mathbf{S}(\tau;y_0))d\tau\right|\leq c_1\|y_0\|_{-\beta}
\end{equation}
for any $y_0\in V^{-\beta}(\mathbb{T}^3)$ and $t>0$, where $\Phi$
is the functional defined in \eqref{Phi_def} and $S(t;y_0)$ is the
solution operator of problem
\eqref{3Dheat_eq}-\eqref{3Dheat_in_cond}.\footnote{Here and below
we use for brevity notation $S(t;y_0)$ instead of $S(t,\cdot
;y_0)$.}
\end{lemma}
\begin{proof}
The solution $\mathbf{S}(t, \cdot;y_0)$ of \eqref{3Dheat_eq}-\eqref{3Dheat_in_cond} can be represented as
\begin{equation*}
\mathbf{S}(t,\cdot;y_0)=\sum_{k\neq 0}\hat{y}_0(k)\cdot \mathrm{e}^{\mathrm{i}(k,\cdot)}\cdot\mathrm{e}^{-|k|^2t},
\end{equation*}
where $\hat{y}_0(k)$ are the Fourier coefficients of function $y_0$. Therefore, according to definition \eqref{sobolev_space_norm} and \eqref{Phi_prop1},
\begin{equation}\label{Phi_prop2_1}
{\left|\int_0^t\Phi(\mathbf{S}(\tau,\cdot;y_0))d\tau\right|}\leq c\int_0^t \mathrm{e}^{-\tau/2}\left(\sum_{k\neq 0}(|\hat{y}_0(k)|^2|k|^{-2\beta})|k|^{3+2\beta}\mathrm{e}^{-(k^2-1)\tau}\right)^{1/2}d\tau.
\end{equation}

For every $t>0$ let us consider the following extremal problem:
\begin{equation*}
f_t(x)=x^{3+2\beta}\mathrm{e}^{-(x^2-1)t}\rightarrow \max, \,\, x\geq 1.
\end{equation*}

Its solution is $\hat{x}=\sqrt{\frac{3+2\beta}{2t}}$. So,
\begin{equation}\label{f_t_max}
\max(f_t(x))=\begin{cases}
             \left(\frac{3+2\beta}{2t}\right)^{\frac{3+2\beta}{2}}\mathrm{e}^{-(3+2\beta-2t)/2}, &t\leq\frac{3+2\beta}{2},\\
             1, & t\geq\frac{3+2\beta}{2}.
             \end{cases}
\end{equation}
 Substituting \eqref{f_t_max} into \eqref{Phi_prop2_1}, we arrive at \eqref{Phi_prop2}.
\end{proof}

According to Lemma \ref{Phi_prop2_lem}, the functional in the left
hand side of \eqref{Phi_prop2} is well-defined for $y_0\in
V^{-\beta}(\mathbb{T}^3)$ with $\beta<1/2$. In particular, Lemma
\ref{Phi_prop2_lem} and explicit formula \eqref{exp_sol} show,
that the solution of problem \eqref{npe}, \eqref{NPE_boundcond},
\eqref{Helmholtz_incond} is well-defined for any initial data
$y_0\in V^0$ and is infinitely differentiable for every
$x\in\mathbb{T}^3$ and $t\in(0,T)$, where $T$ depends on the
choice of $y_0$.

In the following two sections we justify our choice of $V^0$ as the phase space of the corresponding dynamical system.

\subsection{Unique solvability  of NPE}

Let $Q_T=(0,T)\times \mathbb{T}^3,\;T>0$ or $T=\infty $. The
following space of solutions for NPE will be used:
$$
    V^{1,2(-1)}(Q_T)=L_2(0,T;V^1)\cap H^1(0,T;V^{-1})
$$
We look for solutions $\omega (t,x;\omega_0)$ satisfying

\begin{condition}\label{ex cond} If initial condition $\omega_0\in V^0\setminus
\{ 0\}$ and solution $\omega (t,x;\omega_0)\in V^{1,2(-1)}(Q_T)$
then $ \omega(t,\cdot;\omega_0)\ne 0\; \forall t\in [0,T]$
\end{condition}

\begin{theorem} For each $\omega_0 \in V^0$ there exists $T>0$ such
that there exists unique solution $\omega (t,x;\omega_0)\in
V^{1,2(-1)}(Q_T)$ of the problem
\eqref{npe},\eqref{NPE_boundcond}, \eqref{Helmholtz_incond}
satisfying Condition {{\ref{ex cond}}}
\end{theorem}

\begin{theorem} The solution $\omega (t,x;\omega_0)\in
V^{1,2(-1)}(Q_T)$ of the problem
\eqref{npe},\eqref{NPE_boundcond}, \eqref{Helmholtz_incond}
depends continuously on initial condition $\omega_0\in V^0$.
\end{theorem}

One can see the proof of these Theorems in \cite{F4}.

Below we will use the following specification of existence theorem
formulated above for small initial conditions.
\begin{lemma}\label{sol_est_th0}
There exists $r_0>0$ such that for each $\omega_0\in
B_{r_0}=\{y\in V^0:{{\|y\|_{0}}}\leq r_0\}$ the solution
$\omega(t,x;\omega_0)$ of problem
\eqref{npe},\eqref{NPE_boundcond}, \eqref{Helmholtz_incond}
 satisfies
\begin{equation}\label{sol_est0}
{\|\omega(t,\cdot;\omega_0)\|_{0}}\leq c_0 \mathrm{e}^{-t}, \text{ as }
t\rightarrow \infty.
\end{equation}
\end{lemma}
\begin{proof}
In virtue of explicit formula \eqref{exp_sol} and Lemma
\ref{Phi_prop2_lem},
\begin{equation}\label{sol_est01}
{\|\omega(t,\cdot;\omega_0)\|_{0}}\leq\frac{{\|\mathbf{S}(t,\cdot;\omega_0)\|_{0}}}{1-\int_0^t\Phi(\mathbf{S}(\tau,\cdot;\omega_0))d\tau}
\leq\frac{{\|\omega_0\|_{0}}\mathrm{e}^{-t}}{1-c_1{\|\omega_0\|_{0}}},
\end{equation}
where $c_1$ is the constant from estimate \eqref{Phi_prop2}. So,
taking in \eqref{sol_est01} ${\|\omega_0\|_{0}}=r_0=\frac{1}{2c_1}$, we
get the bound \eqref{sol_est0} with $c_0=\frac{1}{c_1}$.
\end{proof}

%%%%%%%%%%%%%%%%%%%%%%%%%%%%%%%%%%%%%%%%%%%%%%%%%%%%%%%%%%%%%%%%%%%%%%%%%
\subsection{Structure of dynamical flow for NPE}

We will use $V^0(\mathbb{T}^3)\equiv V^0$ as the phase space for
problem \eqref{npe},\eqref{NPE_boundcond},
\eqref{Helmholtz_incond}.

\begin{definition}\label{M-} The set $M_-\subset V^0$ of $\omega_0$, such that the
corresponding solution $\omega(t,x;\omega_0)$ of problem
\eqref{npe},\eqref{NPE_boundcond}, \eqref{Helmholtz_incond}
satisfies inequality
$$
      \|\omega(t,\cdot ;\omega_0)\|_0\le \alpha \|\omega_0\|_0\mathrm{e}^{-t}\qquad \forall
      t>0
$$
is called the set of stability. Here $\alpha >1$ is a fixed number
depending on $\| \omega_0\|_0$.
\end{definition}

\begin{definition}\label{M+} The set $M_+\subset V^0$ of $\omega_0$, such that the
corresponding solution $\omega(t,x;\omega_0)$ exists only on a finite
time interval $t\in (0,t_0)$, and blows up at $t=t_0$ is called
the set of explosions.
\end{definition}

\begin{definition}\label{Mg} The set $M_g\subset V^0$ of $\omega_0$, such that the corresponding
solution ${\omega(t,x;\omega_0)}$ exists for time $t\in \mathbb{R}_+$, and
$\|\omega(t,x;\omega_0)\|_{0}\to \infty$ as $t\to \infty$ is called the
set of growing.
\end{definition}

\begin{lemma}(see \cite{F4}) Sets $M_-,M_+,M_g$ are not empty, and $M_-\cup M_+\cup
M_g=V^0$
\end{lemma}
%%%%%%%%%%%%%%%%%%%%%%%%%%%%%%%%%%%%%%%%%%%%%%%%%%%%%%%%%%%%%%%
\subsection{On a geometrical structure of phase space}

Let define the following subsets of unit sphere: $\Sigma =\{ v\in
V^0:\; \| v\|_0=1\}$ in the phase space $V^0$:
$$
  A_-(t)=\{v\in \Sigma : \int_0^t\Phi (\mathbf{S}(\tau,\cdot; v))d\tau \le 0\},
  \quad A_-=\cap_{t\ge 0}A_-(t),
$$
$$
 B_+=\Sigma \setminus A_-\equiv
 \{ v\in \Sigma : \; \exists t_0>0 \; \int_0^{t_0}\Phi (S(\tau, \cdot; v))d\tau >0\},
$$
$$
   \partial B_+=\{ v\in \Sigma :\; \forall t>0 \int_0^{t}\Phi (\mathbf{S}(\tau ,\cdot;v))d\tau \le 0 \quad
 \mbox{and}\; \exists t_0>0: \; \int_0^{t_0}\Phi (\mathbf{S}(\tau,\cdot;v))d\tau =0 \}
$$
We introduce the following function on sphere $\Sigma $:
\begin{equation}\label{func sphere}
   B_+\ni v\to b(v)=\max_{t\ge 0} \int_0^{t}\Phi (\mathbf{S}(\tau,\cdot;v))d\tau
\end{equation}
Evidently, $b(v)>0$ and  $b(v)\to 0$ as $v\to \partial B_+$. Let
define the map $\Gamma (v)$:
\begin{equation}\label{map}
   B_+\ni v\to \Gamma (v)=\frac{1}{b(v)}v\in V^0
\end{equation}
It is clear that $\| \Gamma (v)\|_0\to \infty $ as $v\to \partial
B_+$. The set $\Gamma(B_+)$ divides $V^0$ into {two
parts: the set $V^0_-$ of points $v \in V^0\setminus \Gamma (B_+)$ lying on the same side relative to hypersurface $\Gamma (B_+)$ as the origin $0$ of $V^0$, and the set $V^0_+=V^0\setminus \{V^0_-\cup \Gamma(B_+)\}$. Since the map $\Gamma (v), v\in B_+$ is defined on subset $B_+$ of unit sphere $\Sigma \subset V^0$, we can define sets  $V^0_-, V^0_+$ as follows:}

$$
   V^0_-=\{ v\in V^0:\; [0,v]\cap \Gamma(B_+)=\emptyset \},
$$
$$
   V^0_+=\{ v\in V^0:\; [0,v)\cap \Gamma(B_+)\ne \emptyset \}
$$
Let $B_+=B_{+,f}\cup B_{+,\infty}$ where
$$
   B_{+,f}=\{ v\in B_+:\; \mbox{max in \eqref{func sphere} is achived at}\; t<\infty \}
$$
$$
   B_{+,\infty}=\{ v\in B_+:\; \mbox{max in \eqref{func sphere} is not achived at}\; t<\infty \}
$$
\begin{theorem}(see \cite{F4}) $M_-=V^0_-,\; M_+=V^0_+\cup B_{+,f},\;
M_g=B_{+,\infty}$
\end{theorem}

%%%%%%%%%%%%%%%%%%%%%%%%%%%%%%%%%%%%%%%%%%%%%%%%%%%%%%%%%%%%%%%%%%%%%%%%%%%%%%%%%
\section{Stabilization of solution for NPE by starting control}\label{s2}

%%%%%%%%%%%%%%%%%%%%%%%%%%%%%%%%%%%%%%%%%%%%%%%%%%%%%%%%%%%%%%%%%%%%%%%%%%%%%%%%%%%%

\subsection{Formulation of the main result on stabilization}

We consider semilinear parabolic equations \eqref{npe}:
\begin{equation}\label{NPE_1}
    \partial_t{y}(t,{x})-\Delta {y}(t,{x})-\Phi ({y}){y}=0
\end{equation}
 with periodic boundary condition
\begin{equation}\label{npe_bound}
   {y}(t,...x_i+2\pi,... )={y}(t,{x}),\, i=1,2,3
\end{equation}
and initial condition
\begin{equation}\label{npe_in_contr}
    {y}(t,{x})|_{t=0}={y}_0({x})+{u}_0({x}).
\end{equation}
Here $\Phi$ is the functional defined in \eqref{Phi_def}, ${y}_0(
{x})\in V^0(\mathbb{T}^3)$ is an arbitrary given initial datum and
${u}_0({x})\in V^0(\mathbb{T}^3)$ is a control. Phase space $V^0$
is defined in \eqref{phase_space}.

We assume that ${u}_0({x})$ is supported on
$[a_1,b_1]\times[a_2,b_2]\times[a_3,b_3]\subset \mathbb{T}^3 =
(\mathbb{R}/2\pi\mathbb{Z})^3$:
\begin{equation}\label{supp_v}
    \mbox{supp}\,{u}_0\subset [a_1,b_1]\times[a_2,b_2]\times[a_3,b_3]
\end{equation}

Our goal is to find for every given ${y}_0(x)\in
V^0(\mathbb{T}^3)$ a control ${u}_0\in V^0(\mathbb{T}^3)$
satisfying \eqref{supp_v} such that there exists unique solution
${y}(t,{x};{y}_0+{u}_0)$ of \eqref{NPE_1}-\eqref{npe_in_contr} and
this solution satisfies the estimate
\begin{equation}\label{stab_est}
    \| {y}(t,\cdot ;{y}_0+{u}_0)\|_0 \le \alpha \|{y}_0+{u}_0\|_0\mathrm{e}^{-t}\quad \forall t>0
\end{equation}
with a certain $\alpha >1$.

By Definition \ref{M-} of the set of stability $M_-$ inclusion
$y_0\in M_-$ implies estimate \eqref{stab_est} with $u_0=0$.
Therefore the formulated problem is reach of content only if
$y_0\in V^0\setminus M_-=M_+\cup M_g$. Note, that without loss of generality, the last inclusion
can be changed for $y_0\in V^{1/2}\setminus M_-$. Indeed, in virtue of explicit formula \eqref{exp_sol} the solution $y_0(t,\cdot;y_0)$ of NPE belongs to $C^{\infty}(\mathbb{T}^3)$ for arbitrary small $t>0$. Hence, if $y_0\in V^0$, we can shift on small $t$, take $y(t,\cdot;y_0)$ {as} initial condition and apply the stabilization construction to it.

The following main theorem holds:

\begin{theorem}\label{stabiliz_theorem}
Let ${y}_0\in V^{1/2}\setminus M_-$ be given. Then there exists a
control ${u}_0\in V^0\cap (L_\infty (\mathbb{T}^3))^3$ satisfying
\eqref{supp_v} such that there exists a unique solution
${y}(t,{x};{y}_0+{u}_0)$ of \eqref{NPE_1}--\eqref{npe_in_contr},
and this solution satisfies bound \eqref{stab_est} with a certain
$\alpha >1$.
\end{theorem}

The main steps of this theorem's proof are indicated below.
%%%%%%%%%%%%%%%%%%%%%%%%%%%%%%%%%%%%%%%%%%%%%%%%%%%%%%%%%%%%%%%%%%%%%%%%%%%%%%%%%%%%%%%%%%

\subsection{Formulation of the main preliminary result}

To rewrite condition \eqref{supp_v} in more convenient form, let us
first perform the change of variables in
\eqref{NPE_1}-\eqref{npe_in_contr}:
$$\tilde{x}_i=x_i - \frac{a_i+b_i}{2}, i=1,2,3$$
and denote

\begin{equation}\label{newvar}
\begin{split}
&\tilde {{y}}(t,\tilde {{x}})= {{y}}\left(t,\tilde {x}_1+\frac{a_1+b_1}{2},\tilde {x}_2+\frac{a_2+b_2}{2},\tilde {x_3}+\frac{a_3+b_3}{2}\right),\\
&\tilde {{y}}_0(\tilde {{x}})= {{y}}_0\left(\tilde {x}_1+\frac{a_1+b_1}{2},\tilde {x}_2+\frac{a_2+b_2}{2},\tilde {x}_3+\frac{a_3+b_3}{2}\right),\\
&\tilde {{u}_0}(\tilde {{x}})= {{u}_0}\left(\tilde
{x}_1+\frac{a_1+b_1}{2},\tilde {x}_2+\frac{a_2+b_2}{2},\tilde
{x}_3+\frac{a_3+b_3}{2}\right).
\end{split}
\end{equation}

Then substituting \eqref{newvar} into relations \eqref{NPE_1}-\eqref{npe_in_contr}, \eqref{stab_est} and omitting the tilde sign leaves these relations unchanged, while inclusion \eqref{supp_v} transforms into
\begin{equation}\label{supp_v_p}
\mbox{supp}\,{u}_0\subset
[-\rho_1,\rho_1]\times[-\rho_2,\rho_2]\times[-\rho_3,\rho_3]
\end{equation}
where $\rho_i = \dfrac{b_i-a_i}{2}\in(0,\pi)$, $i=1,2,3$.

Below we consider stabilization problem \eqref{NPE_1}-\eqref{npe_in_contr}, \eqref{stab_est}
with condition \eqref{supp_v_p} instead of \eqref{supp_v}.

We look for a starting control $u_0(x)$ in a form
\begin{equation}\label{StartContr}
  u_0(x)=-\lambda u(x)
\end{equation}
where the constant $\lambda >0$ will be
defined later and the main component $u(x)$ is defined as follows.
For given $\rho_1,\rho_2,\rho_3\in(0,\pi)$ we choose $p\in \mathbb
N$ such that
\begin{equation}\label{p_def}
\frac{\pi}{p}\le \rho_i, \,i=1,2,3,
\end{equation}
and denote by $\chi_{\frac{\pi}{p}}(\alpha )$ the characteristic
function of interval $(-\frac{\pi}{p},\frac{\pi}{p})$:
\begin{equation}\label{def_chi}
\chi_{\frac{\pi}{p}}(\alpha) = \left\{\begin{array}{rl}
        1, & |\alpha | \leq \frac{\pi}{p},\\
        0, & \frac{\pi}{p}<|\alpha | \leq \pi.
            \end{array}\right.
\end{equation}
Then we set
\begin{equation}\label{u_def}
{u}({x})=\frac{\tilde u(x)}{\| \tilde u\|_0}\quad \mbox{with}\;
\tilde u(x) =
\operatorname{curl}\operatorname{curl}(\chi_{\frac{\pi}{p}}(x_1)\chi_{\frac{\pi}{p}}(x_2)\chi_{\frac{\pi}{p}}(x_3)w(px_1,px_2,px_3),0,0),
\end{equation}
where
\begin{equation}\label{def_w}
w(x_1,x_2,x_3) =\sum_{\genfrac{}{}{0pt}{}{i,j,k=1}{i<j,k\neq i,j}}^{3}a_k(1+\cos x_k)(\sin x_i+\frac{1}{2}\sin 2x_i)(\sin x_j+\frac{1}{2}\sin 2x_j),
\end{equation}
$a_1,a_2,a_3\in \mathbb{R}$.

\begin{proposition}\label{prop_u} The vector field $u(x)$ defined in
\eqref{p_def}-\eqref{def_w} possesses the following properties:
\begin{equation}\label{main_est_a}
   {u}({x})\in V^0(\mathbb{T}^3)\cap(L_\infty (\mathbb{T}^3))^3,\qquad \mbox{supp}\,u\, \subset ([-\rho
   ,\rho])^3, \qquad \| u\|_0=1
\end{equation}
\end{proposition}

\begin{proof}
For each $j=1,2,3$ function $w(x_1,x_2,x_3)$ defined in
\eqref{def_w}
 and $\partial_jw$ equal to zero at $x_j=\pm \pi$. That is why
using notations $\mathbf{\chi}_{\frac{\pi}{p}}(x)=
\chi_{\frac{\pi}{p}}(x_1)\chi_{\frac{\pi}{p}}(x_2)\chi_{\frac{\pi}{p}}(x_3),\,
w(px)=w(px_1,px_2,px_3)$ we get
\begin{equation}\label{def_cw}
\mbox{curl}(\mathbf{\chi}_{\frac{\pi}{p}}(x)(w(px),0,0))=p\mathbf
{\chi}_{\frac{\pi}{p}}(x)(0,\partial_3w(px),-\partial_2w(px))\in
(H^1(\mathbb{T}^3))^3
\end{equation}
\begin{equation}\label{def_ccw}
  u(x)=p^2{\mathbf
{\chi}}_{\frac{\pi}{p}}(x)(-\partial_{22}w(px)-\partial_{33}w(px),\partial_{12}w(px),\partial_{13}w(px))\in
(H^0(\mathbb{T}^3))^3
\end{equation}
Applying to vector field \eqref{def_ccw} operator \mbox{div} and
performing  direct calculations in the space of distributions we
get that $\mbox{div}\, u(x)=0$. Hence, $u(x)\in
V^0(\mathbb{T}^3)$. The other relations in \eqref{main_est_a} are
evident.
\end{proof}

Let consider the boundary value problem for the system of three
heat equations
\begin{equation}\label{heat_eq_u}
    \partial_t{\mathbf{S}}(t,{x};u)-\Delta {\mathbf{S}}(t,{x};u)=0,\qquad  \mathbf{S}(t,{x})|_{t=0}={u}({x})
\end{equation}
with periodic boundary condition. (Since by Proposition
\ref{prop_u} $\mbox{div}\,u(x)=0$ we get that $\mbox{div}\,
S(t,x;u)=0$ for $t>0$, and therefore system \eqref{heat_eq_u} in
fact is equal to the Stokes system.)

The following theorem is true:
\begin{theorem}\label{main_est_th}
For each $\rho :=\pi /p \in (0,\pi )$ the function ${u}({x})$
defined in \eqref{u_def} by a natural number $p$ satisfying
\eqref{p_def} and characteristic function \eqref{def_chi},
satisfies the estimate:
\begin{equation}\label{main_est_b}
\int_{T^3}((\mathbf{S}(t,{x};{u}),\nabla)\operatorname{curl}^{-1}\mathbf{S}(t,{x};{u}),\mathbf{S}(t,{x};{u}))dx>3\beta
\mathrm{e}^{-18t} \qquad \forall \; t\ge 0
\end{equation}
with a positive constant $\beta$.
\end{theorem}

The proof of Theorem \ref{main_est_th} was given in \cite{FSh16}. This theorem is the most complicated part of Theorem's \ref{stabiliz_theorem} proof.

 \begin{remark}
 Actually, estimate {\eqref{main_est_b}} in {\cite{FSh16}} was proved not for the function $u(x)$, defined in {\eqref{u_def}}, but for the function $\tilde{u}(x)$, defined in the same formula {\eqref{u_def}}. But, dividing both parts of {\eqref{main_est_b}} for $\tilde{u}(x)$ on $\|\tilde{u}\|^3$, applying formula {\eqref{u_def}}, expressing $u(x)$ via $\tilde{u}(x)$, to the left hand side of the resulting inequality and denoting $\beta/\|\tilde{u}\|^3$ as $\beta$, we get that estimate {\eqref{main_est_b}} is true for $u(x)$ as well, with a different constant $\beta$. Everywhere below we use estimate {\eqref{main_est_b}} for function $u(x)$, defined in {\eqref{u_def}}, i.e. we hold $\|u\|:=\|u\|_{V^0(\mathbb{T}^3)}=1$.
\end{remark}

\subsection{ Proof of the stabilization result: the first step}\label{s2.41}

In this subsection we begin to prove Theorem
\ref{stabiliz_theorem} using Theorem \ref{main_est_th}. We take
control \eqref{StartContr} as a desired one where vector-function
$u(x)$ is defined in \eqref{u_def}, \eqref{def_w}, and $\lambda
\gg 1$ is a parameter.

In fact, to prove the stabilization result, it is enough to show
that at some instant $t_0$ the solution $y(t,x;y_0+u_0)$ of the
stabilization problem \eqref{NPE_1}, \eqref{npe_bound},
\eqref{npe_in_contr} belongs to a small enough neighborhood of
zero. This is implied by Lemma \ref{sol_est_th0}.

So, to prove the desired result we need to show, that for every
$T>0$ we can choose parameter $\lambda$ in such way that the
function
\begin{equation}\label{denominator}
    1-\int_0^t\Phi (\mathbf{S}(\tau ,\cdot ;y_0-\lambda u))d\tau
\end{equation}
for each $t\in(0,T)$ is bounded from below by a positive constant
independent of $t$, and that is why the solution
$y(t,x;y_0-\lambda u)$ of the stabilization problem \eqref{NPE_1},
\eqref{npe_bound}, \eqref{npe_in_contr} satisfies estimate
$\|y(T,\cdot;y_0-\lambda u)\|\leq r_0:= \frac{1}{2c_1}$ (see Lemma
\ref{sol_est_th0}).

We begin with the proof of one important corollary of Theorem
\ref{main_est_th}. Denote
\begin{equation}\label{Psi}
  \Psi(y_1,y_2,y_3)=\int_{T^3}((y_1,\nabla)\mbox{curl}^{-1}y_2,y_3)dx,
  \quad
  \Psi(y)=\Psi(y,y,y)
\end{equation}

\begin{lemma}\label{denom_est_lem1}
Let $u(x)$ be the control function from Theorem \ref{main_est_th},
$\mathbf{S}(t ,\cdot ;y_0-\lambda u)$ be the solution of the
Stokes system \eqref{heat_eq_u} with initial condition
 $\mathbf{S}(t,x;y_0-\lambda u)|_{t=0}=y_0-\lambda u$. Then
 for every $T>0$ there exists $\lambda\gg 1$, such that for every $t\in(0,T)$
 the following estimate holds:
\begin{equation}\label{psi_est}
-\Psi(\mathbf{S}(t,\cdot;y_0-\lambda u)) > 2\beta\lambda^3
\mathrm{e}^{-18t},
\end{equation}
where $\beta$ is the constant from \eqref{main_est_b}.
\end{lemma}

\begin{proof}
According to Theorem \ref{main_est_th},
\begin{equation}\label{psi_u_est}
\Psi(\mathbf{S}(t,\cdot;u)) \geq 3\beta \mathrm{e}^{-18t}, \, \beta >0.
\end{equation}

From definition \eqref{Psi} of $\Psi$,
\begin{equation}\label{psi_est_proof1}
\begin{split}
&-\Psi (\mathbf{S}(t,y_0-\lambda u))=\lambda^3\Psi(\mathbf{S}(t,u))-
\lambda^2(\Psi(\mathbf{S}(t,u),\mathbf{S}(t,u),\mathbf{S}(t,y_0))+\\
&\Psi(\mathbf{S}(t,u),\mathbf{S}(t,y_0),\mathbf{S}(t,u))+
\Psi(\mathbf{S}(t,y_0),\mathbf{S}(t,u),\mathbf{S}(t,u)))+\\
&\lambda(\Psi(\mathbf{S}(t,u),\mathbf{S}(t,y_0),\mathbf{S}(t,y_0))+
\Psi(\mathbf{S}(t,y_0),\mathbf{S}(t,u),\mathbf{S}(t,y_0))+\\
&\Psi(\mathbf{S}(t,y_0),\mathbf{S}(t,y_0),\mathbf{S}(t,u)))-\Psi(\mathbf{S}(t,y_0)).
\end{split}
\end{equation}

In virtue of well-known estimate for pseudo-differential
operators (see \cite{Es}),  Sobolev embedding theorem and definition \eqref{Psi}
we get
\begin{equation}\label{Psi_est1}
\begin{split}
&|\Psi (y_1,y_2,y_3)|\leq\| y_1\|_{L_3(\mathbb{T}^3)}\|\nabla
     \mbox{curl}^{-1}y_2\|_{L_3}\|y_3\|_{L_3}\leq \\
&\tilde{c}_3\|y_1\|_{L_3(\mathbb{T}^3)}\|y_2\|_{L_3}\|y_3\|_{L_3}\le
c_3 \|y_1\|_{V^{1/2}}\|y_2\|_{V^{1/2}}\|y_3\|_{V^{1/2}}
\end{split}
\end{equation}

According to \eqref{psi_u_est}-\eqref{Psi_est1},
\begin{equation}\label{psi_est_proof2}
\begin{split}
&-\Psi(\mathbf{S}(t,y_0-\lambda u))>3\beta \lambda^3\mathrm{e}^{-18t}-\\
&c\left(\lambda^2\|\mathbf{S}(t,u)\|^2_{V^{1/2}}\|
\mathbf{S}(t,y_0)\|_{V^{1/2}}+\lambda
 \|\mathbf{S}(t,u)\|_{V^{1/2}}\| \mathbf{S}(t,y_0)\|^2_{V^{1/2}}+\|\mathbf{S}(t,y_0)\|^3_{V^{1/2}}\right)
\end{split}
\end{equation}
By the same way, using \eqref{psi_u_est}-\eqref{Psi_est1} and
inequality $\|v\|_{L_3(\mathbb{T}^3)}\le c
\|v\|_{L_\infty(\mathbb{T}^3)}$ we get
\begin{equation}\label{psi_est_proof3}
\begin{split}
&-\Psi(\mathbf{S}(t,y_0-\lambda u))>3\beta \lambda^3\mathrm{e}^{-18t}-\\
&c\left(\lambda^2\|\mathbf{S}(t,u)\|^2_{L_{\infty}}\|
\mathbf{S}(t,y_0)\|_{V^{1/2}}+\lambda
 \|\mathbf{S}(t,u)\|_{L_{\infty}}\| \mathbf{S}(t,y_0)\|^2_{V^{1/2}}+\|\mathbf{S}(t,y_0)\|^3_{V^{1/2}}\right).
\end{split}
\end{equation}

Let us show, that for small enough $t_0>0$
\begin{equation}\label{S_u_est1}
\|\mathbf{S}(t,u)\|_{V^{1/2}}\leq\frac{\mathrm{e}^{t_0-1/4}}{\sqrt{2}t_0^{1/4}}\|u\|_{V^0}\mathrm{e}^{-t}:=
A_{t_0}\|u\|_{V^0}\mathrm{e}^{-t}, \,\, \forall t>t_0,
\end{equation}
where the last equality is the definition of the constant $A_{t_0}$, depending on $t_0$.

Indeed, for a fixed small enough $t_0>0$
\begin{equation}\label{S_u_est2}
\|\mathbf{S}(t,u)\|^2_{V^{1/2}}=\sum_{k\in\mathbb{Z}^3\setminus
\{0\}}|\hat{u}_k|^2|k|\mathrm{e}^{-2k^2t_0}\mathrm{e}^{-2k^2(t-t_0)}\leq\|S(t_0,u)\|^2_{V^{1/2}}\mathrm{e}^{-2(t-t_0)}\,\,
\forall t>t_0.
\end{equation}
Next, since function $f_{t_0}(y)=y\mathrm{e}^{-2y^2t_0},\, y\in
\mathbb{R}_+$ reaches its maximum at
$\hat{y}=\frac{1}{2\sqrt{t_0}}$ and
$f_{t_0}(\hat{y})=\frac{\mathrm{e}^{-1/2}}{2\sqrt{t_0}}$,
\begin{equation}\label{S_u_est3}
\|\mathbf{S}(t_0,u)\|^2_{V^{1/2}}=\sum_{k\in\mathbb{Z}^3\setminus
\{0\}}|\hat{u}_k|^2|k|\mathrm{e}^{-2k^2t_0}\leq
\frac{\mathrm{e}^{-1/2}}{2\sqrt{t_0}}\|u\|^2_{V^0}.
\end{equation}
Estimate \eqref{S_u_est1} follows from \eqref{S_u_est2}-\eqref{S_u_est3}.

%Using the maximum principle for a heat equation, we obtain
%\begin{equation}\label{S_u_est4}
%\|\mathbf{S}(t,u)\|_{L_{\infty}}\leq\|u\|_{L_{\infty}}, t\in(0,t_0).
%\end{equation}

Estimates \eqref{psi_est_proof2} and \eqref{S_u_est1} together with the estimate
\begin{equation}\label{S_y_est}
\|\mathbf{S}(t,y_0)\|_{V^{1/2}}\le \|y_0\|_{V^{1/2}}\mathrm{e}^{-t},
\end{equation}
imply, that for all $t\in(t_0,T)$
\begin{equation*}
\begin{split}
-&\Psi(\mathbf{S}(t,y_0-\lambda u))>\\
&3\beta\lambda^3\mathrm{e}^{-18t}
-c\mathrm{e}^{-3t}(\lambda^2 A^2_{t_0}\|y_0\|_{V^{1/2}}\!+\!\lambda A_{t_0}\|y_0\|^2_{V^{1/2}}\!+\!\|y_0\|^3_{V^{1/2}})\ge \\
&2\beta\lambda^3 \mathrm{e}^{-18t}+\lambda^3 \mathrm{e}^{-3t}\left(\beta \mathrm{e}^{-15T}\vphantom{\frac{\|y_0\|^2_{V^{1/2}}}{\lambda^2}}-\right.\\
&\left.c(A^2_{t_0}\frac{\|y_0\|_{V^{1/2}}}{\lambda}+A_{t_0}\frac{\|y_0\|^2_{V^{1/2}}}{\lambda^2}+
\frac{\|y_0\|^3_{V^{1/2}}}{\lambda^3})\right)>\\
&2\beta\lambda^3 \mathrm{e}^{-18t}+\lambda^3 \mathrm{e}^{-3t}\left(\beta \mathrm{e}^{-15T}-\frac{c}{\lambda}(A^2_{t_0}\|y_0\|_{V^{1/2}}+
A_{t_0}\|y_0\|^2_{V^{1/2}}+\|y_0\|^3_{V^{1/2}})\right).\\
\end{split}
\end{equation*}
Therefore, choosing for every  $T>0$
\begin{equation}\label{lambda_01_def}
\lambda>\lambda_{01}=\frac{c\mathrm{e}^{15T}}{\beta}(A^2_{t_0}\|y_0\|_{V^{1/2}}+A_{t_0}\|y_0\|^2_{V^{1/2}}+\|y_0\|^3_{V^{1/2}}),
\end{equation}
we get the bound \eqref{psi_est} for all $t\in(t_0,T)$.

Now let us show, that for small enough $t_0$ this relation also holds for $t\in(0,t_0)$.

Using the maximum principle for a heat equation, we obtain
\begin{equation}\label{S_u_est4}
\|\mathbf{S}(t,u)\|_{L_{\infty}}\leq\|u\|_{L_{\infty}}, t\in(0,t_0).
\end{equation}

Therefore, according to \eqref{psi_est_proof3} and \eqref{S_y_est}, we get that $\forall t\in(0,t_0)$
\begin{equation*}
\begin{split}
&-\Psi(\mathbf{S}(t,y_0-\lambda u))>\\
&3\beta\lambda^3\mathrm{e}^{-18t}
-c(\lambda^2\|u\|^2_{L_{\infty}}\|y_0\|_{V^{1/2}}\mathrm{e}^{-t}+
\lambda\|u\|^2_{L_{\infty}}\|y_0\|^2_{V^{1/2}}\mathrm{e}^{-2t}\!+\!\|y_0\|^3_{V^{1/2}}\mathrm{e}^{-3t}).
\end{split}
\end{equation*}

Choosing
$t_0<\frac{\|u\|^4_{V_0}}{4\mathrm{e}\|u\|^4_{L_{\infty}}}=\frac{1}{4\mathrm{e}\|u\|^4_{L_{\infty}}}$,
we get $\|u\|_{L_{\infty}}<\frac{\mathrm{e}^{-1/4}}{\sqrt{2}t_0^{1/4}}=A_{t_0}\mathrm{e}^{-t_0}$,
and
\begin{equation*}
\begin{split}
&-\Psi(\mathbf{S}(t,y_0-\lambda u))>\\
&3\beta\lambda^3\mathrm{e}^{-18t}
-\\
&c(\lambda^2A_{t_0}^2\mathrm{e}^{-2t_0}\|y_0\|_{V^{1/2}}\mathrm{e}^{-t}+\lambda A_{t_0}\mathrm{e}^{-t_0}\|y_0\|^2_{V^{1/2}}\mathrm{e}^{-2t}\!+\!\|y_0\|^3_{V^{1/2}}\mathrm{e}^{-3t})>\\
&2\beta\lambda^3\mathrm{e}^{-18t}+\\
&\lambda^3\mathrm{e}^{-18t_0}\left(\beta
-c\left(A_{t_0}^2\frac{\|y_0\|_{V^{1/2}}}{\lambda}\mathrm{e}^{16t_0}+A_{t_0}\frac{\|y_0\|^2_{V^{1/2}}}{\lambda^2}\mathrm{e}^{17t_0}+
\!\frac{\|y_0\|^3_{V^{1/2}}}{\lambda^3}\mathrm{e}^{18t_0}\right)\right)>\\
&2\beta\lambda^3\mathrm{e}^{-18t}+\\
&\lambda^3\mathrm{e}^{-18t_0}\left(\beta-
\frac{c}{\lambda}\left(A_{t_0}^2\|y_0\|_{V^{1/2}}\mathrm{e}^{16t_0}+A_{t_0}\|y_0\|^2_{V^{1/2}}\mathrm{e}^{17t_0}+\!\|y_0\|^3_{V^{1/2}}\mathrm{e}^{18t_0}\right)\right).
\end{split}
\end{equation*}

The last estimate implies, that if
\begin{equation}\label{lambda_02_def}
\lambda>\lambda_{02}:=\frac{c}{\beta}\left(A_{t_0}^2\|y_0\|_{V^{1/2}}\mathrm{e}^{16t_0}+A_{t_0}\|y_0\|^2_{V^{1/2}}\mathrm{e}^{17t_0}+\!\|y_0\|^3_{V^{1/2}}\mathrm{e}^{18t_0}\right),
\end{equation}
then $-\Psi(\mathbf{S}(t,y_0-\lambda u))>2\beta\lambda^3\mathrm{e}^{-18t}$ $\forall t\in(0,t_0)$.

Finally, taking $\lambda>\lambda_0:=\max\{\lambda_{01},\lambda_{02}\}$, where $\lambda_{01}$ and $\lambda_{02}$ were defined in \eqref{lambda_01_def} and \eqref{lambda_02_def} correspondingly, we get that the estimate \eqref{psi_est} is true for all $t\in(0,T)$.
%The denominator of $\Phi (\mathbf{S}(\tau ,\cdot ;y_0-\lambda u))$
%is positive, i.e.
%\begin{equation}\label{denom_est}
%\int_{\mathbb{T}^3}|\mathbf{S}(t,y_0-\lambda u)|^2dx>0
%\end{equation}
%Bounds \eqref{psi_est},\eqref{denom_est} imply that the function
%\eqref{denominator} is more than 1.
\end{proof}
%\section{Some development of stabilization theory of solution for NPE}\label{s2.42}

Further we shall need estimate \eqref{psi_est} in a more convenient form:

\begin{proposition}\label{psi_est_lem2}
 Let $\Psi(y_1,y_2,y_3)$ be the function defined in \eqref{Psi}, and $u(x)$ -- the control function \eqref{u_def}-\eqref{def_w}. Then for each $y_0\in V^{1/2}\setminus M_{-}$ and for each $T>0$ there exists $\lambda_0=\lambda(\|y_0\|_{V^{1/2}},T)$ such that for any $\lambda>\lambda_0$
\begin{equation}\label{psi_est1}
\frac{-\Psi(\mathbf{S}(t,\cdot;y_0-\lambda
u))}{\|\mathbf{S}(t,\cdot;y_0-\lambda u)\|^3_{V^0}} > \beta
\mathrm{e}^{-15t},\quad \forall t\in(0,T),
\end{equation}
where $\beta$ is the positive constant from Theorem
\ref{main_est_b}.
\end{proposition}
\begin{proof}
Below we shall use notation $\|\cdot\|=\|\cdot\|_{V^0}$. Since
$\|u\|=1$, we have
\begin{equation}
2=2\|u\|^3\geq 2\left(\|u-\frac{y_0}{\lambda}\|-\|\frac{y_0}{\lambda}\|\right)^3=\|u-\frac{y_0}{\lambda}\|^3+B,
\end{equation}
where
\begin{equation}
\begin{split}
&B=\|u-\frac{y_0}{\lambda}\|^3-6\|u-\frac{y_0}{\lambda}\|^2\|\frac{y_0}{\lambda}\|+6\|u-\frac{y_0}{\lambda}\|\|\frac{y_0}{\lambda}\|^2
-2\|\frac{y_0}{\lambda}\|^3=\\
&\|u-\frac{y_0}{\lambda}\|^2\left(\|\|u-\frac{y_0}{\lambda}\|-6\frac{y_0}{\lambda}\|\right)+
6\|\frac{y_0}{\lambda}\|^2\left(\|u-\frac{y_0}{\lambda}\|-\frac{1}{3}\|\frac{y_0}{\lambda}\|\right)\geq\\
&\|u-\frac{y_0}{\lambda}\|^2(1-7\|\frac{y_0}{\lambda}\|)+6\|\frac{y_0}{\lambda}\|^2(1-\frac{4}{3}\|\frac{y_0}{\lambda}\|)>0
\end{split}
\end{equation}
for $\lambda>7\|y_0\|$. Therefore, $2>\|u-\frac{y_0}{\lambda}\|^3$. Applying this inequality to the right hand side of \eqref{psi_est} and dividing both parts be $\lambda^3\|u-\frac{y_0}{\lambda}\|^3$, we get that
\begin{equation}\label{psi_est2}
\frac{-\Psi(\mathbf{S}(t,\cdot;y_0-\lambda u))}{\|y_0-\lambda u\|^3} > \beta \mathrm{e}^{-18t},\quad \forall t\in(0,T),
\end{equation}

Similarly to estimate \eqref{S_u_est2},
\begin{equation}\label{psi_est1_S_u}
\|S(t,\cdot;y_0-\lambda u)\|^2=\sum_{k\in\mathbb{Z}^3\setminus
\{0\}}|\hat{y}_{0,k}-\lambda
\hat{u}_k|^2\mathrm{e}^{-2k^2t}\leq\|y_0-\lambda u\|^2 \mathrm{e}^{-2t}\,\, \forall
t>0.
\end{equation}
 Dividing $-\Psi(\mathbf{S}(t,\cdot;y_0-\lambda u))$ by $\|\mathbf S(t,y_0-\lambda u)\|^3$, taking into account \eqref{psi_est2} and \eqref{psi_est1_S_u}, we obtain
\begin{equation*}
\frac{-\Psi(\mathbf{S}(t,\cdot;y_0-\lambda u))}{\|\mathbf
S(t,y_0-\lambda u)\|^3} >
\frac{-\Psi(\mathbf{S}(t,\cdot;y_0-\lambda u))}{\|y_0-\lambda
u\|^3\mathrm{e}^{-3t}}>\beta \mathrm{e}^{-15t},\quad \forall t\in(0,T).
\end{equation*}
This completes the proof of \eqref{psi_est1}.
\end{proof}
\subsection{ Proof of the stabilization result: the second step}

This section completes the proof of Theorem \ref{stabiliz_theorem}.

First, let us prove the following:
\begin{theorem}\label{sol_est_th}
 Let $y(t,x;v)$ be the solution of \eqref{npe}, \eqref{NPE_boundcond} with initial condition
 $v:=y_0-\lambda u$, $y_0\in V^{1/2}\setminus M_-$, $\lambda\gg 1$, where $u(x)$ is the control
 function defined in \eqref{u_def}--\eqref{def_w}. Then for every $T>0$ there exists such
 $\lambda_0=\lambda(\|y_0\|_{V^{1/2}},T)$, that for every $\lambda>\lambda_0$ $y(t,x;v)$ satisfies
 the following inequality:
\begin{equation}\label{sol_est1}
\|y(t,\cdot;v)\|\leq\frac{\|v\|\mathrm{e}^{-t}}{1+\frac{\beta}{16}\|v\|(1-\mathrm{e}^{-16t})}, \quad \forall t\in(0,T),
\end{equation}
where $\beta$ it the positive constant from Theorem \ref{main_est_b}.
\end{theorem}
\begin{proof}
Multiplying equation \eqref{NPE_1} by $y(t,x;v)$ scalarly in $V^0$
and taking into account definitions \eqref{Phi_def}, \eqref{Psi}
of functionals $\Phi$ and $\Psi$, we get after simple
transformation:
\begin{equation*}
\frac{1}{2}\partial_t\|y(t,v)\|^2+\|y_x(t,v)\|^2-\Psi(y(t,v))=0,
\end{equation*}
where $y_x=(\partial_{x_1}y,\partial_{x_2}y,\partial_{x_3}y,)$.
Dividing this equality on $\|y(t,v)\|^3$, we obtain that
\begin{equation*}
\frac{\partial_t\|y(t,v)\|}{\|y(t,v)\|^2}+\frac{\|y_x(t,v)\|^2}{\|y(t,v)\|^3}=\frac{\Psi(y(t,v))}{\|y(t,v)\|^3}.
\end{equation*}
Let us introduce notation $z(t)=1/\|y(t,v)\|$. Then the last equality can be rewritten as
\begin{equation}\label{sol_est_th_ineq1}
-\partial_t
z(t)+z(t)=\frac{\Psi(y(t,v))}{\|y(t,v)\|^3}-\frac{\|y_x(t,v)\|^2}{\|y(t,v)\|^3}+
\frac{1}{\|y(t,v)\|}\leq\frac{\Psi(y(t,v))}{\|y(t,v)\|^3},
\end{equation}
because
$$
-\frac{\|y_x(t,v)\|^2}{\|y(t,v)\|^3}+
\frac{1}{\|y(t,v)\|}=-\frac{1}{\|y(t,v)\|}\left( \frac{\sum_{k\ne
0}k^2|\hat v_k|^2\mathrm{e}^{-k^2t}}{\sum_{k\ne 0}|\hat v_k|^2\mathrm{e}^{-k^2t}}-1
\right)\le 0.
$$
Let us transform the right side of \eqref{sol_est_th_ineq1}. In
virtue of explicit formula \eqref{exp_sol},
\begin{equation}\label{sol_est_th_ineq2}
\|y(t,v)\|^3\leq\frac{\|\mathbf{S}(t,\cdot;v)\|^3}{{\left(1-\int_0^t\frac{\Psi(\mathbf{S}(\tau,\cdot ;v)}{\|\mathbf{S}(\tau,\cdot;v)\|^2}d\tau\right)^3}},
-\Psi(y(t,\cdot;v))=\frac{-\Psi(\mathbf{S}(t,\cdot; v))}{{\left(1-\int_0^t\frac{\Psi(\mathbf{S}(\tau,\cdot ;v)}{\|\mathbf{S}(\tau,\cdot;v)\|^2}d\tau\right)^3}}
\end{equation}
Dividing the second relation of \eqref{sol_est_th_ineq2} on the first one, we get that
\begin{equation}\label{sol_est_th_ineq3}
\frac{{-\Psi(y(t,\cdot; v))}}{\|y(t,v)\|^3}\geq \frac{{-\Psi(\mathbf{S}(t,\cdot; v))}}{\|\mathbf{S}(t,\cdot; v)\|^3}.
\end{equation}

Inequalities \eqref{sol_est_th_ineq1}, \eqref{sol_est_th_ineq3} and \eqref{psi_est1} imply the estimate
\begin{equation}\label{sol_est_th_ineq11}
\partial_t z(t)-z(t)\geq\beta \mathrm{e}^{-15t}, \text{ or } z(t)\geq \mathrm{e}^{t}\left(\frac{1}{\|v\|}+\frac{\beta}{16}(1-\mathrm{e}^{-16t})\right). \end{equation}
Changing back from $z$ to $y$, we get \eqref{sol_est1}.
\end{proof}
Now, let us prove the following corollary of Theorem \ref{sol_est_th}:
\begin{cor}\label{cor_fin}
Let $y(t,x;v)$ be the solution of \eqref{npe}, \eqref{NPE_boundcond} with initial condition $v:=y_0-\lambda u$,
$y_0\in V^{1/2}\setminus M_-$, $\lambda\gg 1$, where $u(x)$ is the control function defined in
\eqref{u_def}--\eqref{def_w}. Then there exists $T>0$ independent of $v$, such that
$y(t,\cdot;v)\in B_{r_0}$, where $r_0=\frac{1}{2c_1}$ was defined in Lemma \ref{sol_est_th0}.
\end{cor}
\begin{proof}
According to Theorem \ref{sol_est_th}, for every $T>0$ we can choose $\lambda_0$ in such a way,
that for every $\lambda>\lambda_0$ the solution $y(t,x;y_0-\lambda u)$ of the problem
\eqref{npe}, \eqref{NPE_boundcond} with initial condition $y_0-\lambda u$ satisfies the estimate
\begin{equation*}
\|y(t,\cdot;v)\|\leq\frac{\|v\|\mathrm{e}^{-t}}{1+\frac{\beta}{16}\|v\|(1-\mathrm{e}^{-16t})}<
\frac{16\mathrm{e}^{-t}}{\beta(1-\mathrm{e}^{-16t})},  \quad \forall t\in(0,T).
\end{equation*}

Taking into account this bound, we should find $T$ satisfying
\begin{equation*}
\frac{16\mathrm{e}^{-T}}{\beta(1-\mathrm{e}^{-16T})}\leq r_0.
\end{equation*}

Denoting $\mathrm{e}^{-T}=:x$ and changing $r_0$ for $\frac{1}{2c_1}$, we
reduce the problem to finding the roots from $(0,1)$ of the
following equation:
\begin{equation*}
F(x):=\beta x^{16}+32c_1x-\beta =0.
\end{equation*}

Since $F(0)=-\beta<0$, $F(1)=32c_1>0$ and $F'(x)>0$ for
$x\in(0,1)$, this equation has a unique solution $x_0\in(0,1)$.
Therefore, for $T=\ln \frac{1}{x_0}$ we get that $y(T,\cdot; v)\in
B_{r_0}$ with $r_0=\frac{1}{c_1}$.
\end{proof}

{\it The completion of Theorem \ref{stabiliz_theorem} proof.} Let
$T$ be the instant calculated in Corollary \ref{cor_fin}, and
$\lambda_0:=\lambda_0(\| y_0\|_{V^1/2},T)$ be the function from
formulation of Theorem \ref{sol_est_th}. By this Theorem for each
$\lambda
>\lambda_0$ the solution $y(t,x;y_0-\lambda u)$ of problem \eqref{npe}, \eqref{NPE_boundcond}
with initial condition $y_0-\lambda u$ satisfies estimate
\eqref{sol_est1}. Hence by Corollary \ref{cor_fin} $\|y(T,\cdot
;y_0-\lambda u)\|_0\le r_0$, where $r_0$ is the radius of the ball
from Lemma \ref{sol_est_th0}. By this lemma the solution
$y(t,x;y_0-\lambda u)$ tends to zero exponentially as $T< t\to
\infty$. This completes the justification of the stabilization
construction.
%%%%%%%%%%%%%%%%%%%%%%%%%%%%%%%%%%%%%%%%%%%%%%%%%%%%%%%
%%% Acknowledgements. пїЅпїЅР»
%%%%%%%%%%%%%%%%%%%%%%%%%%%%%%%%%%%%%%%%%%%%%%%%%%%%%%%

%%%%%%%%%%%%%%%%%%%%%%%%%%%%%%%%%%%%%%%%%%%%%%%%%%%%%%%
%%% Conflict of interest. пїЅпїЅпїЅпїЅпїЅпїЅпїЅпїЅпїЅпїЅпїЅпїЅ
%%%%%%%%%%%%%%%%%%%%%%%%%%%%%%%%%%%%%%%%%%%%%%%%%%%%%%%
%\InterestConflict

%%%%%%%%%%%%%%%%%%%%%%%%%%%%%%%%%%%%%%%%%%%%%%%%%%%%%%%
%%% Supplements. пїЅпїЅпїЅпїЅпїЅпїЅпїЅпїЅ, пїЅЗ±пїЅСЎ
%%%%%%%%%%%%%%%%%%%%%%%%%%%%%%%%%%%%%%%%%%%%%%%%%%%%%%%
%\Supplements{}

%%%%%%%%%%%%%%%%%%%%%%%%%%%%%%%%%%%%%%%%%%%%%%%%%%%%%%%
%%% Reference section. пїЅОїпїЅпїЅпїЅпїЅпїЅ
%%% citation in the content using "some words~\cite{1,2}".
%%% ~ is needed to make the reference number is on the same line with the word before it.
%%%%%%%%%%%%%%%%%%%%%%%%%%%%%%%%%%%%%%%%%%%%%%%%%%%%%%%


\begin{thebibliography}{99}
\bibitem{BLT}
\newblock V.Barbu, I.Lasiecka, R.Triggiani,
\newblock{Abstract setting of tangential boundary stabilization
of Navier-Stokes equations by high-andlow-gain feedback
controllers},
\newblock \emph{Nonlinear Analysis}, \textbf{64} (2006), 2704--2746.

\bibitem{B}
\newblock M.Badra,
\newblock{Abstract setting for stabilization of nonlinear parabolic system with a Riccaty-based
  strategy. Application to Navier-Stokes an Boussinesq equations with Neumann or Dirichlet
  control},
\newblock \emph{Discrete and Continuous Dynamical Systems}, \textbf{32}, no.4 (2012), 1169--1208.

\bibitem{C1}
\newblock J.M. Coron,
\newblock On null asymptotic stabilization of the two-dimensional incompressible
Euler equations in a simply connected domains,
\newblock \emph{SIAM J.Control Optim.}, \textbf{37} (1999), 1874--1896.

\bibitem{C2}
\newblock J.M. Coron,
\newblock{Control and Nonlinearity},
\newblock \emph{Math.Surveys and Monographs}, \textbf{136},
\newblock AMS, Providence, RI, 2007, 426 p.


\bibitem{CF}
\newblock J.M.Coron, A.V.Fursikov,
\newblock {Global exact controllability of the 2D Navier-Stokes equations
on manifold without boundary}
\newblock \emph{J.Russian Math. Phys.}, \textbf{4} (1996), 1--20

\bibitem{F2}
\newblock A.V.Fursikov,
\newblock{The simplest semilinear parabolic equation of normal
type,}
\newblock \emph{Mathematical Control and Related  Fields(MCRF)}, \textbf{2} (2012),  141--170


\bibitem{F4}
\newblock A.V.Fursikov,
\newblock{On parabolic system of normal type corresponding to 3D Helmholtz system,}
\newblock \emph{Advances in Mathematical Analysis of PDEs, Proc. St. Petersburg Math. Soc.},
\textbf{XV}; AMS Transl.Series 2 \textbf{232} (2014),  99--118

\bibitem{F5}
\newblock A.V.Fursikov,
\newblock{Stabilization of the simplest normal parabolic equation by starting control.}
\newblock \emph{Communications on
 pure and applied analysis}, \textbf{13} (2014),  1815--1854

 \bibitem{F6}
 \newblock A.V.Fursikov,
 \newblock{Stabilization for the 3D Navier-Stokes system by
feedback boundary control},
\newblock \emph{Discrete and Cont. Dyn. Syst.}, \textbf{10} (2004), 289--314.

\bibitem{FG}
\newblock A.V.Fursikov, A.V.Gorshkov,
\newblock{Certain questions of feedback stabilization for Navier-Stokes
equations.}
\newblock \emph{Evolution equations and control theory (EECT)}, \textbf{1} (2012), 109--140


\bibitem{FI}
\newblock A.V.Fursikov, O.Yu Immanuvilov,
\newblock {Exact controllability of Navier-Stokes and Boussinesq equations.}
\newblock \emph{Russian Math. Survveys}, \textbf{54} (1999), 565--618

\bibitem{FK}
\newblock A.V.Fursikov, A.A.Kornev,
\newblock{Feedback stabilization for Navier-Stokes equations:
theory and calculations.}
\newblock \emph{Mathematical Aspects of Fluid
Mechanics (LMS Lecture Notes Series)}, \textbf{402}, Cambridge University
Press, (2012),  130--172



\bibitem{FSh}
\newblock A.V.Fursikov, L.S.Shatina,
\newblock{On an estimate related to the stabilization on a normal
parabolic equation by starting control},
\newblock \emph{Journal of Mathematical Sciences},
 \textbf{217}, № 6, (2016), 803--826

\bibitem{FSh16}
\newblock A.V.Fursikov, L.S.Shatina,
\newblock{Nonlocal stabilization of the normal equation connected with Helmholtz system by starting control},
\newblock\emph{Discrete and Continuous Dynamical Systems (DCDS-A)}, \textbf{38}, no.3 (2018), 1187--1242.

\bibitem{F7}
\newblock A.V.Fursikov,
\newblock{Normal equation generated from Helmholtz system: nonlocal stabilization by starting
control and properties of stabilized solutions},
\newblock accepted for publication in  \emph{Recent developments in Integrable Systems and related topics of Mathematical Physics (eds. V. M. Buchstaber et al.)}, PROMS, Springer (2018)

%G. Eskin, Lectures on Linear Partial Differential Equations, Amer. Math. Society,
%Providence RI, 2011
\bibitem{Es}
\newblock G. Eskin,
\newblock \emph{Lectures on Linear Partial Differential Equations,}
\newblock Amer. Math. Society, Providence RI, (2011), 410p.

\bibitem{K}
\newblock M.Krstic,
\newblock{On global stabilizationof Burgers' equation by boundary control},
\newblock \emph{Systems of control letters}, \textbf{37} (1999), 123--141.


\bibitem{R1}
\newblock J.-P.Raymond,
\newblock{Feedback boundary stabilization of the three-dimensional incompressible
Navier-Stokes equations},
\newblock \emph{J. Math. Pures Appl.}, \textbf{ 87} (2007), 627--669.

\bibitem{RTh}
\newblock J.-P.Raymond, L.Thevenet,
\newblock{Boundary feedback stabilization of the two-dimensional
Navier-Stokes equations with final dimensional controllers},
\newblock \emph{Discrete and Continuous Dynamical Systems (DCDS-A)}, \textbf{ 27}, no.3 (2010), 1159--1187.

\end{thebibliography}
\end{document}